\documentclass{article}

\usepackage{arxiv}

\usepackage[utf8]{inputenc} 
\usepackage[T1]{fontenc}    
\usepackage{hyperref}       
\usepackage{url}            
\usepackage{booktabs}       
\usepackage{amsfonts}       
\usepackage{nicefrac}       
\usepackage{microtype}      
\usepackage{lipsum}		
\usepackage{graphicx}
\usepackage{doi}
\usepackage{amsthm}
\usepackage[square,sort,comma,numbers]{natbib}

\usepackage{amssymb,amsthm,amsmath}
\usepackage{xcolor,paralist,hyperref,titlesec,fancyhdr,etoolbox}

\newtheorem{lm}{Lemma}[section]
\newtheorem{cor}{Corollary}[section]
\newtheorem{pro}{Proposition}[section]

\title{On the tightness of linear relaxations of alternative mixed integer programming formulations for the generator maintenance scheduling problem}

\date{February 12, 2025}	

\author{ \href{https://orcid.org/0000-0000-0000-0000}{\includegraphics[scale=0.06]{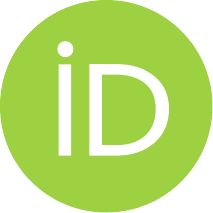}\hspace{1mm}Tiago~Andrade}\thanks{The author thanks his colleagues at PSR for the valuable discussions on generator maintenance scheduling problem and for reviewing the manuscript, especially Sergio Granville, Gerson Couto, Luiz Carlos da Costa Junior, Rodrigo Benoliel, and Alessandro Soares.} \\
	PSR, Rio de Janeiro, Brazil\\
	\texttt{tiago.andrade@psr-inc.com} \\
}

\renewcommand{\shorttitle}{On the tightness of linear relaxations of alternative MIP formulations for the GMS problem}

\hypersetup{
pdftitle={On the tightness of linear relaxations of alternative mixed integer programming formulations for the generator maintenance scheduling problem},
pdfsubject={},
pdfauthor={Tiago Andrade},
pdfkeywords={Generator maintenance scheduling, mixed integer programming, linear relaxation},
}

\begin{document}
\maketitle

\begin{abstract}
	This paper presents a comprehensive theoretical analysis of six distinct Mixed-Integer Programming (MIP) formulations for preventive Generator Maintenance Scheduling (GMS), a critical problem for ensuring the reliability and efficiency of power systems. By comparing the tightness of their linear relaxations, we identify which formulations offer superior dual bound and, thus, better computational performance. Our analysis includes establishing relationships between the formulations through definitions, lemmas, and propositions, demonstrating that some formulations provide tighter relaxations that lead to more efficient optimization outcomes. These findings offer valuable insights for practitioners and researchers in selecting the most effective models to enhance the scheduling process of preventive generator maintenance.
\end{abstract}


\section{Introduction} \label{section
}

The efficient and reliable operation of modern power systems critically depends on the effective scheduling of maintenance for generating units. Generators, as the primary sources of electricity production, operate under continuous stress and are subject to wear and degradation over time. Periodic maintenance is essential not only to restore their operational efficacy but also to prevent potential failures that could cascade through the power grid, leading to widespread outages and significant economic losses.

Generator Maintenance Scheduling (GMS) is a complex and multifaceted problem that requires strategic planning and careful consideration of various operational constraints. The challenge lies in determining optimal maintenance periods that balance the need for equipment upkeep with the obligation to meet fluctuating electricity demand, ensure adequate reserve margins, and comply with regulatory standards. Poorly scheduled maintenance can result in increased operational costs, reliance on expensive backup power sources, compromised system reliability, and potential regulatory non-compliance.

Over the years, numerous studies have developed and refined MIP formulations for GMS, aiming to enhance model accuracy and computational efficiency such as \cite{froger2016maintenance,chen1991optimal,kralj1995multiobjective,fourcade1997optimizing,moro1999goal,conejo2005generation,wu2008genco,barot2008security,mazidi2018profit,zhang2023extended,niazi2013long,norozpour2013preventive}.
While these MIP formulations are designed to model the same GMS problem under the same operational assumptions, they differ in how they mathematically represent scheduling decisions and constraints. Each formulation introduces its own set of variables and constraints to capture the maintenance scheduling process. However, these differences in formulation can lead to variations in the tightness of their linear relaxations -- the solutions obtained when the integrality of decision variables is relaxed to allow fractional values.

The tightness of a linear relaxation is a crucial factor affecting the performance of optimization algorithms used to solve MIP problems, such as branch-and-bound and branch-and-cut. A tighter linear relaxation provides a better approximation of the integer solution, potentially reducing the computational effort required to reach optimality. Despite modeling the same problem under the same hypothesis, the alternative formulations can yield linear relaxations that vary in tightness, impacting computational efficiency and solution quality.

However, existing literature often does not justify choosing one formulation over another, nor does it thoroughly compare the computational implications of different formulations. This research gap motivates the need for a systematic analysis of these formulations to understand their relative strengths and weaknesses in terms of tightness of linear relaxation. 

In this context, our paper aims to provide a comprehensive theoretical analysis of several MIP formulations for preventive generator maintenance scheduling. We present and formalize six distinct formulations, each capturing the essential aspects of GMS under the same operational assumptions. By focusing on the differences in their linear relaxations, we aim to identify which formulations offer tighter relaxations and therefore may lead to more efficient computational performance when solving the GMS problem.

There are other variations of GMS that focus on other parts of the problem, such as predictive maintenance, where, using generator data, the planner tries to find out when it can still operate and when it must enter maintenance. Or, for example, a decision of maintenance planning counting how many equivalent hours of operation the given generator has before it must enter a maintenance state. These variations are not in the scope of this paper. We are focusing on models that adopt the same core premises. That is, we are planning for preventive generator maintenance with a given maintenance list with fixed durations that must be scheduled in a given horizon. Furthermore, when maintenance starts, the plant must remain in a maintenance state until its completion. That is, each maintenance cannot be divided. We also do not consider specific constraints such as precedence or non-coincidence constraints. These constraints could be, in principle, added to any of the studied formulations and would only unnecessarily complicate the intended analysis. We focus on analyzing the representation of the maintenance state and how to ensure the given duration and continuity of maintenance. 

Our theoretical comparisons include definitions, lemmas, and propositions that establish the relationships between different formulations. We demonstrate that specific formulations are strictly tighter than others, meaning that their linear relaxations provide better approximations of the integer solutions, which can lead to more efficient computational performance. These insights are valuable for practitioners who aim to enhance both computational efficiency and solution quality in preventive generator maintenance scheduling since a stronger formulation usually achieves a better solution quicker and provides a better bound faster.

\section{Literature review}

Several comprehensive surveys on maintenance scheduling within the energy industry have been published. Notably, \cite{froger2016maintenance} presents a systematic review of maintenance scheduling models and methods, while \cite{prajapat2017preventive} offers another extensive overview of preventive maintenance scheduling approaches.

One of the earlier studies on optimal maintenance planning for power systems is presented by Chen and Toyoda \cite{chen1991optimal}. The authors formulate two distinct problems: (i) an isolated-system maintenance scheduling model that ignores network constraints, solved via linear programming and branch-and-bound (BnB); and (ii) a multi-area system model that incorporates network constraints, tackled with a decomposition technique based on virtual load. Their objective is to maximize the minimum reserve margin -— defined as (available capacity - demand)/demand -- while also forbidding units within the same power plant from overlapping their maintenance periods.

Kralj and Petrovic \cite{kralj1995multiobjective} subsequently extend maintenance scheduling to a multi-objective (MO) framework, focusing on thermal units. They employ a multi-objective BnB algorithm with successive approximations, targeting multiple criteria such as minimizing fuel costs, maximizing reliability, and reducing constraint violations. Their approach is illustrated through a realistic example involving 21 thermal generating units.



Focusing on nuclear power plants, Fourcade et al.\ \cite{fourcade1997optimizing} develop a generation maintenance scheduling model for up to four reactors over a five-year horizon. Their objective is to minimize total fuel cost (nuclear and fossil) under the constraint that at most one reactor can be shut down each week. Moro and Ramos \cite{moro1999goal} propose a goal-programming-based approach that captures both economic and reliability criteria. In their two-stage solution, they first solve a purely economic objective (minimizing operating costs), then add a constraint bounding these costs near the optimal level while seeking to improve reliability through reserve margin criteria. Their model, which includes resources like pumped-hydro units and fuel consumption limits, is applied to the weekly maintenance scheduling of the Spanish power system.

Shifting attention to competitive electricity markets, Conejo et al.\ \cite{conejo2005generation} discuss a coordinating mechanism for generation maintenance scheduling that balances producer incentives with system-level reliability requirements. In a similar vein, Wu et al.\ \cite{wu2008genco} propose a stochastic risk-based model for generation companies (GENCOs), integrating hourly price-based unit commitment with Monte Carlo simulations to handle price uncertainty in energy, ancillary services, and fuel costs. Barot and Bhattacharya \cite{barot2008security} address security-constrained maintenance planning in deregulated markets by allowing independent system operators (ISOs) to coordinate and incentivize maintenance schedules without directly imposing them, thereby reducing unserved energy at individual buses.

More recently, Beloniel et al.~\cite{benoliel2024considering} address the impact of utilizing an optimal maintenance scheduling in the hydrothermal medium term dispatch and compare it to utilizing an approximation using the plant's historical average unavailabilities. They also evaluate the effect of uncertainty in the maintenance executions. Furthermore, they propose a CVaR objective function to not only consider the maximization of minimum reserve margin as the standard practice.


\section{Mathematical Model}



\subsection{Formulation 1}

The binary variable $X_{t,m}$ indicates whether the maintenance $m$ is executed in period $t$. It values one if that is the case and zero otherwise. Binary variable $S_{t,m}$ indicates whether the maintenance $m$ begins executing in period $t$. It values one if that is the case and zero otherwise. The same notation is used hereinafter for other formulations too.

Constraint~\eqref{eq:1a} imposes that the total of periods that a given maintenance $m$ is executed sums to the maintenance duration $W_m$. Constraint~\eqref{eq:1b} imposes that each maintenance begins only once. Note that, although not necessarily explicit when it is used in the literature, each maintenance solution mustn't initiate at the end of the horizon such that it would not terminate in the horizon, i.e., $ S_{t,m} = 0, \forall m \in M, t \geq |T| - W_m + 2$. Constraint~\eqref{eq:1c} together with the previous \eqref{eq:1a} and \eqref{eq:1b} impose that the maintenance is uninterrupted until it ends. It also links variables $S$ and $X$.

\cite{niazi2013long,norozpour2013preventive} use this formulation.

\begin{subequations}
\label{eq:1}
\begin{align}
    & \sum_{t \in T} X_{t,m}  = W_m & \forall m \in M \label{eq:1a} \\
    & \sum_{t \in T} S_{t,m}  = 1 & \forall m \in M \label{eq:1b} \\
    & X_{t,m} - X_{t-1,m} \leq S_{t,m}  & \forall m \in M \label{eq:1c} \\
    & X_{t,m} \in \{0,1\} & \quad \forall t \in T, m \in M \label{eq:1d}  \\
    & S_{t,m} \in \{0,1\} & \forall t \in T, m \in M \label{eq:1e}
\end{align}
\end{subequations}


\subsection{Formulation 2}

Formulation~\eqref{eq:2} only uses variable $X$ indicating when the maintenance is being executed, not when it started directly. Therefore, it uses half of the number of variables of Formulation~\eqref{eq:1}.

\cite{zhang2023extended} uses this formulation.

Constraint~\eqref{eq:2a} is the same as Constraint~\eqref{eq:1a}. Constraint~\eqref{eq:2b} is the constraint that imposes non-preemptive scheduling. It can be interpreted as if the maintenance $m$ is being executed at period $t$. Then, the same maintenance must be executed at least the time period $t-1$ or at the following $W_m$ periods. Note that it is implicitly defined that $X_{0,m}=0$ and that $X_{tl,m}=0$ for any $tl$ that is beyond the end of the considered horizon.

\begin{subequations}
\label{eq:2}
\begin{align}
    & \sum_{t \in T} X_{t,m}  = W_m & \forall m \in M \label{eq:2a} \\
    & X_{t,m} - X_{t-1,m} \leq X_{tl,m} & \forall t \in T, m \in M, tl \in \{t,...t+W_m-1\} \label{eq:2b} \\
    & X_{t,m} \in \{0,1\} & \quad \forall t \in T, m \in M \label{eq:2c}
\end{align}
\end{subequations}


\subsection{Formulation 3}

Formulation~\eqref{eq:3} also only uses variable $X$ to represent the maintenance schedule. Constraint~\eqref{eq:3a} is the same as Constraint~\eqref{eq:1a}. Constraint~\eqref{eq:3b} is the constraint that imposes non-preemptive scheduling. It can be interpreted as if the maintenance $m$ is executed at time period $t$. Then, the same maintenance must be executed at least the time period $t-1$ or at $t+W_m-1$. Note that it is implicitly defined that $X_{0,m}=0$ and that $X_{tl,m}=0$ for any $tl$ that is beyond the end of the considered horizon.

Note that Formulation~\eqref{eq:3} is a variant of Formulation~\eqref{eq:2} but the Constraint~\eqref{eq:3b} replaces the Constraint~\eqref{eq:2b} where \eqref{eq:3b} is \eqref{eq:2b} for $tl=t+W_m-1$. That is, it removes some constraints.

\cite{mazidi2018profit}, \cite{barot2008security}  and \cite{conejo2005generation} use this formulation.

\begin{subequations}
\label{eq:3}
\begin{align}
    & \sum_{t \in T} X_{t,m}  = W_m & \forall m \in M \label{eq:3a} \\
    & X_{t,m} - X_{t-1,m} \leq X_{t+W_m-1,m}  & \forall t \in T, m \in M \label{eq:3b} \\
    & X_{t,m} \in \{0,1\} & \quad \forall t \in T, m \in M \label{eq:3c}
\end{align}
\end{subequations}


\subsection{Formulation 4}

\cite{moro1999goal} use this formulation with a variable representing the end of the maintenance. For a more direct comparison with other formulations, we added the modeling, replacing it with a variable representing the start of maintenance without changing its presentability. Constraint~\eqref{eq:4c} indicates that if maintenance $m$ is being executed in period $t$, it must have started recently.

\begin{subequations}
\label{eq:4}
\begin{align}
    & \sum_{t \in T} X_{t,m}  = W_m & \forall m \in M \label{eq:4a} \\
    & \sum_{t \in T} S_{t,m}  = 1 & \forall m \in M \label{eq:4b} \\
    & \sum_{t' = t-W_m+1}^t S_{t',m} \geq X_{t,m}  & \forall t \in T, m \in M \label{eq:4c} \\
   & X_{t,m} \in \{0,1\} & \quad \forall t \in T, m \in M \label{eq:4d}  \\
    & S_{t,m} \in \{0,1\} & \forall t \in T, m \in M \label{eq:4e}
\end{align}
\end{subequations}


\subsection{Formulation 5}

\cite{wu2008genco} uses this formulation. This formulation also only uses variable $X$. Constraint~\eqref{eq:5b} represents that if for a given $t$, $X_{t,m}=1$ and $X_{t-1,m}=0$, which means that the maintenance is starting, then the sum of the next $W_m$ periods maintenance executions must sum to $W_m$. All $X_{tl,m}$ for the near "future" must also be activated.

\begin{subequations}
\label{eq:5}
\begin{align}
    & \sum_{t \in T} X_{t,m}  = W_m & \forall m \in M \label{eq:5a} \\
    & W_m  (X_{t,m} - X_{t-1,m}) \leq \sum_{t' = t}^{t+W_m - 1} X_{t',m}  & \forall m \in M \label{eq:5b} \\
    & X_{t,m} \in \{0,1\} & \quad \forall t \in T, m \in M \label{eq:5c}
\end{align}
\end{subequations}

\subsection{Formulation 6}

The formulation~\eqref{eq:6} does not represent the execution variable explicitly. Instead, it uses the expression $\sum_{t' = t-W_m+1}^t S_{t',e}$ in place where the execution variable $X_{t,m}$ would appear. Note that, although not necessarily explicit when it is used in the literature, each maintenance solution must initiate before $W_m$ periods before the end of the horizon $T$. Meaning it must terminate within the horizon, i.e., $ S_{t,m} = 0, \forall m \in M, t \geq |T| - W_m + 2$.

\cite{fourcade1997optimizing} and \cite{benoliel2024considering} use this Formulation.

\begin{subequations}
\label{eq:6}
\begin{align}
    & \sum_{t \in T} S_{t,m}  = 1 & \quad \forall m \in M \label{eq:6a} \\
    & S_{t,m} \in \{0,1\} & \quad \forall t \in T, m \in M \label{eq:6b}
\end{align}
\end{subequations}

 It is equivalent to Formulation~\eqref{eq:7} with fewer variables and constraints. These two formulations have identical linear relaxation and will be treated as the same for linear relaxation analysis.
\begin{subequations}
\label{eq:7}
\begin{align}
    & \sum_{t \in T} S_{t,m}  = 1 & \quad \forall m \in M \label{eq:7a}  \\
    & \sum_{t' = t-W_m+1}^t S_{t',m} = X_{t,m} & \quad \forall t \in T, m \in M \label{eq:7b} \\
    & X_{t,m} \in \{0,1\} & \quad \forall t \in T, m \in M \label{eq:7c} \\
    & S_{t,m} \in \{0,1\} & \quad \forall t \in T, m \in M \label{eq:7d}
\end{align}
\end{subequations}


\subsection{Theorical comparison}

In this section, we compare the linear relaxation of the aforementioned formulations.

\textbf{Definition 1:} Given two MIP formulations, A and B, for the same problem. Formulation A is at least as tight as formulation B if all continuous solutions that satisfy A, also satisfy B. 

\textbf{Definition 2:} Given two MIP formulations, A and B, for the same problem. A is equivalent to B if A is at least as tight as B and B is at least as tight as A. 

\textbf{Definition 3:} Given two MIP formulations, A and B, for the same problem. Formulation A is said to be strictly tighter than formulation B if A is at least as tight as B and there is at least one continuous solution that satisfies B that does not satisfy A. 

\textbf{Definition 4:} Given two MIP formulations, A and B, for the same problem. Formulations A and B are said to be noncomparable if there is at least one continuous solution that satisfies A and does not satisfy B and at least one continuous solution that satisfies B and does not satisfy A.

\textbf{Remark 1:} {We do not mention the objective function in our definition and analysis of linear relaxation because we assume that the same solution in the models will result in the same objective function.}

\textbf{Remark 2:} {We always assume a generous interpretation of the investigated formulations. That is, we eliminate binary variables as a prepossessing even thought it is not necessarily mentioned in the papers. In other words, we always let  $ S_{t,m} = 0, \forall m \in M, t \geq |T| - W_m + 2$ when variable $S$ is defined in the problems.}

\begin{lm}
    \label{lm:main}
    Under Formulation~\eqref{eq:6} and Formulation~\eqref{eq:7}, it holds that:
    \[
    \sum_{t \in T} X_{t,m} = W_m \quad \forall m \in M.
    \]
\end{lm}
\begin{proof}
    Consider any (possibly fractional) solution $(S_{t,m}, X_{t,m})$ satisfying Formulation~\eqref{eq:7}, where:
    \begin{align}
        & \sum_{t \in T} S_{t,m} = 1 \quad \forall m \in M, \eqref{eq:7a} \\
        & X_{t,m} = \sum_{t'=t - W_m +1}^{t} S_{t',m} \quad \forall t \in T, m \in M. \eqref{eq:7b}
    \end{align}
    We aim to prove that $\sum_{t \in T} X_{t,m} = W_m$ for all $m \in M$.

    Summing both sides of Equation~\eqref{eq:7b} over $t \in T$, we obtain:
    \begin{equation}
        \sum_{t \in T} X_{t,m} = \sum_{t \in T} \left( \sum_{t'=t - W_m +1}^{t} S_{t',m} \right). \label{eq:sum_X}
    \end{equation}
    Switching the order of summation on the right-hand side of Equation~\eqref{eq:sum_X}:
    \begin{align*}
        \sum_{t \in T} X_{t,m} &= \sum_{t \in T} \left( \sum_{t'=t - W_m +1}^{t} S_{t',m} \right) \\
        &= \sum_{t' \in T} S_{t',m} \left( \sum_{t = t'+W_m -1}^{t'} 1 \right).
    \end{align*}
    Note that for each $S_{t',m}$, the inner sum over $t$ counts $W_m$ periods (assuming $T$ is sufficiently large to accommodate the maintenance duration $W_m$). Therefore:
    \[
    \sum_{t \in T} X_{t,m} = W_m \sum_{t' \in T} S_{t',m}.
    \]
    Using Equation~\eqref{eq:7a}, we have $\sum_{t' \in T} S_{t',m} = 1$, so:
    \[
    \sum_{t \in T} X_{t,m} = W_m \times 1 = W_m.
    \]
    Thus, we have shown that $\sum_{t \in T} X_{t,m} = W_m$ for all $m \in M$, as required.
\end{proof}

\begin{pro}
Formulation~\eqref{eq:7} are strictly tighter than Formulation~\eqref{eq:1}.
\label{pro:1-7}
\end{pro}
\begin{proof}
First, we show that Formulation~\eqref{eq:7} is at least as tight as Formulation~\eqref{eq:1} by demonstrating that any feasible solution to Formulation~\eqref{eq:7} also satisfies all constraints of Formulation~\eqref{eq:1}.

\textbf{Constraints Equivalence and Satisfaction:}
\begin{itemize}
    \item \textbf{Constraints~\eqref{eq:1d} and \eqref{eq:1e}} specify the bounds on variables:
    \[
    0 \leq X_{t,m} \leq 1, \quad 0 \leq S_{t,m} \leq 1 \quad \forall t \in T, m \in M,
    \]
    which are the same in both formulations and are trivially satisfied.
    
    \item \textbf{Constraint~\eqref{eq:1b}}:
    \[
    \sum_{t \in T} S_{t,m} = 1 \quad \forall m \in M,
    \]
    is identical to \textbf{Constraint~\eqref{eq:7a}} in Formulation~\eqref{eq:7}, so it is satisfied.
    
    \item \textbf{Constraint~\eqref{eq:1a}}:
    \[
    \sum_{t \in T} X_{t,m} = W_m \quad \forall m \in M,
    \]
    is satisfied by \textbf{Lemma~\ref{lm:main}} under Formulation~\eqref{eq:7}.
\end{itemize}

\textbf{Satisfying Constraint~\eqref{eq:1c}:} We need to show that:
\[
X_{t,m} - X_{t-1,m} \leq S_{t,m} \quad \forall t \in T, m \in M.
\]
Under Formulation~\eqref{eq:7}, using \textbf{Constraint~\eqref{eq:7b}}:

\[
X_{t,m} - X_{t-1,m} = \left( \sum_{t'=t - W_m +1}^{t} S_{t',m} \right) - \left( \sum_{t'=t - W_m}^{t - 1} S_{t',m} \right) \\
= S_{t,m} - S_{t-W_m,m} \leq S_{t,m}.
\]

Thus,
\[
X_{t,m} - X_{t-1,m} \leq S_{t,m},
\]

which satisfies Constraint~\eqref{eq:1c}.

Therefore, all constraints of Formulation~\eqref{eq:1} are satisfied by any feasible solution of Formulation~\eqref{eq:7}.

\vspace{0.5em}

Next, we show that Formulation~\eqref{eq:6} is strictly tighter by providing a fractional solution that is feasible for the continuous relaxation of Formulation~\eqref{eq:1} but infeasible for Formulation~\eqref{eq:7}.

\textbf{Counterexample:} Consider an instance with $T = \{1, 2, 3, 4\}$ and a maintenance task $m$ with duration $W_m = 2$. Define the fractional solution:
\[
X_{t,m} = \left( \tfrac{1}{3}, \tfrac{2}{3}, \tfrac{2}{3}, \tfrac{1}{3} \right), \quad S_{t,m} = \left( \tfrac{2}{3}, \tfrac{1}{3}, 0, 0 \right).
\]
This solution satisfies the constraints of Formulation~\eqref{eq:1}:
\begin{itemize}
    \item \textbf{Constraint~\eqref{eq:1a}}:
    \[
    \sum_{t \in T} X_{t,m} = \tfrac{1}{3} + \tfrac{2}{3} + \tfrac{2}{3} + \tfrac{1}{3} = 2 = W_m.
    \]
    \item \textbf{Constraint~\eqref{eq:1b}}:
    \[
    \sum_{t \in T} S_{t,m} = \tfrac{2}{3} + \tfrac{1}{3} = 1.
    \]
    \item \textbf{Constraint~\eqref{eq:1c}}:
    \begin{align*}
    X_{1,m} - X_{0,m} &= \tfrac{1}{3} - 0 = \tfrac{1}{3} \leq S_{1,m} = \tfrac{2}{3}, \\
    X_{2,m} - X_{1,m} &= \tfrac{2}{3} - \tfrac{1}{3} = \tfrac{1}{3} \leq S_{2,m} = \tfrac{1}{3}, \\
    X_{3,m} - X_{2,m} &= \tfrac{2}{3} - \tfrac{2}{3} = 0 \leq S_{3,m} = 0, \\
    X_{4,m} - X_{3,m} &= \tfrac{1}{3} - \tfrac{2}{3} = -\tfrac{1}{3} \leq S_{4,m} = 0.
    \end{align*}
    \item \textbf{Constraints~\eqref{eq:1d} and \eqref{eq:1e}}: All variables are within their bounds $[0, 1]$.
\end{itemize}

However, this solution does not satisfy \textbf{Constraint~\eqref{eq:7b}}, which requires:
\[
X_{t,m} = \sum_{t'=t - W_m +1}^{t} S_{t',m} \quad \forall t \in T.
\]
Checking for $t = 2$:
\[
X_{2,m} = \tfrac{2}{3}, \quad \sum_{t'=1}^{2} S_{t',m} = \tfrac{2}{3} + \tfrac{1}{3} = 1 \neq \tfrac{2}{3}.
\]
Similarly, for $t = 3$:
\[
X_{3,m} = \tfrac{2}{3}, \quad \sum_{t'=2}^{3} S_{t',m} = \tfrac{1}{3} + 0 = \tfrac{1}{3} \neq \tfrac{2}{3}.
\]
Therefore, the solution is infeasible for Formulation~\eqref{eq:6}.

This demonstrates that there exists a feasible solution to Formulation~\eqref{eq:1} that is infeasible for Formulation~\eqref{eq:6}, proving that Formulation~\eqref{eq:6} is strictly tighter than Formulation~\eqref{eq:1}.
\end{proof}

\begin{cor}
Formulation~\eqref{eq:6} is strictly tighter than Formulation~\eqref{eq:1}.
\label{cor:1-6}
\end{cor}

Note that for that the counterexample used for the proof of Proposition~\ref{pro:1-7}, if we maintain the provided solution $X=(\frac{1}{3},\frac{2}{3},\frac{2}{3},\frac{1}{3})$, there is a solution $S$ that satisfies \eqref{eq:7}, however, not this particular provided solution $S$ that satisfies only the Formulation~\eqref{eq:1}.

\begin{lm}
    \label{lm:x_t+j}
    Under Formulation~\eqref{eq:6} and Formulation~\eqref{eq:7}, it holds that:
    \[
    X_{t,m} - X_{t-1,m} \leq X_{t+j,m} \quad \forall m \in M, \quad \forall j \in \{0,\dots,W_m - 1\}.
    \]
\end{lm}
\begin{proof}
    From \textbf{Constraint~\eqref{eq:7b}} of Formulation~\eqref{eq:7}, we have:
    \[
    X_{t,m} = \sum_{t' = t - W_m + 1}^{t} S_{t',m} \quad \forall t \in T, m \in M.
    \]
    Therefore, the difference is:
    \begin{align*}
    X_{t,m} - X_{t-1,m} &= \left( \sum_{t' = t - W_m + 1}^{t} S_{t',m} \right) - \left( \sum_{t' = t - W_m }^{t - 1} S_{t',m} \right) \\
    &\leq S_{t,m}.
    \end{align*}
    Since $S_{t',m}$ variables are non-negative, we have:
    \[
    S_{t,m} \leq \sum_{t' = t}^{t + j} S_{t',m} \quad \forall j \in \{0,\dots,W_m - 1\}.
    \]
    Note that for \( j \in \{0,\dots,W_m - 1\} \), the indices satisfy:
    \[
    t - W_m + 1 + j \leq t + j.
    \]
    Thus, the sum \( \sum_{t' = t}^{t + j} S_{t',m} \) is a subset of the terms in \( X_{t + j,m} \):
    \[
    X_{t + j,m} = \sum_{t' = t + j - W_m + 1}^{t + j} S_{t',m} \geq \sum_{t' = t}^{t + j} S_{t',m}.
    \]
    Therefore, combining the inequalities:
    \[
    X_{t,m} - X_{t-1,m} \leq S_{t,m} \leq \sum_{t' = t}^{t + j} S_{t',m} \leq X_{t + j,m}.
    \]
    Hence,
    \[
    X_{t,m} - X_{t-1,m} \leq X_{t + j,m} \quad \forall m \in M, \quad \forall j \in \{0,\dots,W_m - 1\}.
    \]
\end{proof}

\begin{pro}
\label{pro:2-7}
Formulation~\eqref{eq:7} is strictly tighter than Formulation~\eqref{eq:2}.
\end{pro}
\begin{proof}
First, we show that Formulation~\eqref{eq:7} is at least as tight as Formulation~\eqref{eq:2} by demonstrating that any feasible solution to Formulation~\eqref{eq:7} also satisfies all constraints of Formulation~\eqref{eq:2}.

\textbf{Constraints Equivalence and Satisfaction:}
\begin{itemize}
    \item \textbf{Constraint~\eqref{eq:2c}} specifies the bounds on variables:
    \[
    0 \leq X_{t,m} \leq 1 \quad \forall t \in T, m \in M,
    \]
    which is also defined in Constraint~\eqref{eq:7c}
    
    \item \textbf{Constraint~\eqref{eq:2a}}:
    \[
    \sum_{t \in T} X_{t,m} = W_m \quad \forall m \in M,
    \]
    is satisfied by \textbf{Lemma~\ref{lm:main}} under Formulation~\eqref{eq:7}.
\end{itemize}

\textbf{Satisfying Constraint~\eqref{eq:2b}:} Direct consequence of Lemma~\ref{lm:x_t+j}

Thus, all constraints of Formulation~\eqref{eq:2} are satisfied by any feasible solution of Formulation~\eqref{eq:7}.

\vspace{0.5em}

Next, we show that Formulation~\eqref{eq:7} is strictly tighter by providing a fractional solution that is feasible for Formulation~\eqref{eq:2} but infeasible for Formulation~\eqref{eq:6}.

\textbf{Counterexample:} Consider an instance with $T = \{1, 2, \dots, 10\}$ and a maintenance task $m$ with duration $W_m = 3$. Define the fractional solution:
\[
X_{t,m} = 
\begin{cases}
\frac{1}{3}, & \text{for } t = 1, 2, 3, 6, 7, 8, \\
\frac{1}{4}, & \text{for } t = 4, 5, 9, 10.
\end{cases}
\]
This solution satisfies the constraints of Formulation~\eqref{eq:2}:
\begin{itemize}
    \item \textbf{Constraint~\eqref{eq:2a}}:
    \[
    \sum_{t \in T} X_{t,m} = 6 \times \tfrac{1}{3} + 4 \times \tfrac{1}{4} = 2 + 1 = 3 = W_m.
    \]
    \item \textbf{Constraint~\eqref{eq:2b}}:
    For all $t \in T$, we can verify that:
    \[
    X_{t,m} - X_{t-1,m} \leq X_{tl,m}.
    \]
    For example, at $t=1$:
    \[
    X_{1,m} - X_{0,m} = \tfrac{1}{3} - 0 = \tfrac{1}{3} \leq X_{1,m} = \tfrac{1}{3}.
    \]
    \[
    X_{1,m} - X_{0,m} = \tfrac{1}{3} - 0 = \tfrac{1}{3} \leq X_{2,m} = \tfrac{1}{3}.
    \]
    \[
    X_{1,m} - X_{0,m} = \tfrac{1}{3} - 0 = \tfrac{1}{3} \leq X_{3,m} = \tfrac{1}{3}.
    \]
    Similar calculations can be done for other $t$.
    \item \textbf{Constraint~\eqref{eq:2c}}: All variables are within their bounds $[0, 1]$.
\end{itemize}

However, there cannot be a solution $S_{t,m}$ that satisfies both \textbf{Constraint~\eqref{eq:7a}}:
\[
\sum_{t \in T} S_{t,m} = 1,
\]
and \textbf{Constraint~\eqref{eq:7b}}:
\[
X_{t,m} = \sum_{t'=t - W_m +1}^{t} S_{t',m} \quad \forall t \in T,
\]
because the given $X_{t,m}$ values do not correspond to any $S_{t,m}$ satisfying these constraints. Specifically, the varying values of $X_{t,m}$ cannot be expressed as cumulative sums of $S_{t,m}$ over fixed durations in a way that sums $S_{t,m}$ to 1.

Therefore, this fractional solution is feasible for Formulation~\eqref{eq:2} but infeasible for Formulation~\eqref{eq:7}, proving that Formulation~\eqref{eq:7} is strictly tighter than Formulation~\eqref{eq:2}.
\end{proof}

\begin{pro}
\label{pro:2-6}
Formulation~\eqref{eq:6} is strictly tighter than Formulation~\eqref{eq:2}.
\end{pro}

The exact same proof of Proposition~\ref{pro:2-7} can be used to prove Coroloary~\ref{cor:3-7}

\begin{cor}
\label{cor:3-7}
Formulation~\eqref{eq:7} is strictly tighter than Formulation~\eqref{eq:3}.
\end{cor}

\begin{cor}
\label{cor:3-6}
Formulation~\eqref{eq:6} is strictly tighter than Formulation~\eqref{eq:3}.
\end{cor}

\begin{pro}
\label{pro:4-7}
Formulation \eqref{eq:7} is equivalent to Formulation~\eqref{eq:4}.
\end{pro}
\begin{proof}
First, we show that any feasible solution of Formulation~\eqref{eq:7} satisfies all constraints of Formulation~\eqref{eq:4}.

\textbf{Constraints Equivalence:}
\begin{itemize}
    \item \textbf{Constraint~\eqref{eq:4b}} is the same as \textbf{Constraint~\eqref{eq:7a}}:
    \[
    \sum_{t \in T} S_{t,m} = 1 \quad \forall m \in M.
    \]
    \item \textbf{Constraints~\eqref{eq:4d}} and \eqref{eq:4e} are the same as \textbf{Constraints~\eqref{eq:7c}} and \eqref{eq:7d}, specifying variable domains.
\end{itemize}

\textbf{Satisfying Constraint~\eqref{eq:4a}:}  \[
    \sum_{t \in T} X_{t,m} = W_m \quad \forall m \in M,
    \]
    is satisfied by \textbf{Lemma~\ref{lm:main}} under Formulation~\eqref{eq:6} and Formulation~\eqref{eq:7}.

\textbf{Satisfying Constraint~\eqref{eq:4c}:} From \textbf{Constraint~\eqref{eq:7b}}, we have:
\[
\sum_{t'=t - W_m +1}^{t} S_{t',m} = X_{t,m} \geq X_{t,m}.
\]
Thus, \textbf{Constraint~\eqref{eq:4c}} is satisfied.

\vspace{0.5em}

Next, we show that any feasible solution of Formulation~\eqref{eq:4} satisfies all constraints of Formulation~\eqref{eq:7}.

\textbf{Constraints Equivalence:}
\begin{itemize}
    \item \textbf{Constraint~\eqref{eq:7a}} is the same as \textbf{Constraint~\eqref{eq:4b}}.
    \item \textbf{Constraints~\eqref{eq:7c}} and \eqref{eq:7d} are the same as \textbf{Constraints~\eqref{eq:4d}} and \eqref{eq:4e}.
\end{itemize}

\textbf{Satisfying Constraint~\eqref{eq:7b}:} Summing \textbf{Constraint~\eqref{eq:4c}} over $t \in T$, we get:
\[
\sum_{t \in T} \left( \sum_{t'=t - W_m +1}^{t} S_{t',m} - X_{t,m} \right) \geq 0.
\]
Switching the order of summation:
\[
\sum_{t' \in T} S_{t',m} \left( \sum_{t = t'}^{t'+ W_m -1} 1 \right) - \sum_{t \in T} X_{t,m} \geq 0.
\]
Simplifying:
\[
W_m \sum_{t' \in T} S_{t',m} - W_m = W_m \left( \sum_{t' \in T} S_{t',m} - 1 \right) \geq 0.
\]
Using \textbf{Constraint~\eqref{eq:4b}}, $\sum_{t' \in T} S_{t',m} = 1$, so:
\[
\sum_{t \in T} \left( \sum_{t'=t - W_m +1}^{t} S_{t',m} - X_{t,m} \right) = 0
\]
Since the sum of non-negative terms equals zero, each term must be zero:
\[
\sum_{t'=t - W_m +1}^{t} S_{t',m} - X_{t,m} = 0 \quad \forall t \in T, m \in M.
\]
Thus:
\[
\sum_{t'=t - W_m +1}^{t} S_{t',m} = X_{t,m},
\]
which satisfies \textbf{Constraint~\eqref{eq:7b}}.

Therefore, Formulations~\eqref{eq:4} and \eqref{eq:7} are equivalent.
\end{proof}

\begin{cor}
\label{cor:4-6}
Formulation \eqref{eq:6} is equivalent to Formulation~\eqref{eq:4}.
\end{cor}

\begin{pro}
\label{pro:5-7}
Formulation~\eqref{eq:7} is strictly tighter than Formulation~\eqref{eq:5}.
\end{pro}
\begin{proof}
First, we show that Formulation~\eqref{eq:7} is at least as tight as Formulation~\eqref{eq:5} by demonstrating that any feasible solution to Formulation~\eqref{eq:7} also satisfies all constraints of Formulation~\eqref{eq:5}.

\textbf{Constraints Equivalence and Satisfaction:}
\begin{itemize}
    \item \textbf{Constraint~\eqref{eq:5c}} specifies the bounds on variables:
    \[
    0 \leq X_{t,m} \leq 1 \quad \forall t \in T, m \in M,
    \]
    which is also satisfied in Formulation~\eqref{eq:7}.

    \item \textbf{Constraint~\eqref{eq:5a}}:
    \[
    \sum_{t \in T} X_{t,m} = W_m \quad \forall m \in M,
    \]
    is satisfied by \textbf{Lemma~\ref{lm:main}} under Formulation~\eqref{eq:7}.
\end{itemize}

\textbf{Satisfying Constraint~\eqref{eq:5b}:} We need to show that:
\[
W_m \left( X_{t,m} - X_{t-1,m} \right) \leq \sum_{t'=t}^{t+W_m - 1} X_{t',m} \quad \forall t \in T, m \in M.
\]
Using \textbf{Lemma~\ref{lm:x_t+j}}, we have:
\[
X_{t,m} - X_{t-1,m} \leq X_{t+j,m} \quad \forall m \in M, \quad \forall j \in \{0, \dots, W_m - 1\}.
\]
Summing both sides over $j$ from $0$ to $W_m - 1$, we get:
\[
\sum_{j=0}^{W_m - 1} \left( X_{t,m} - X_{t-1,m} \right) \leq \sum_{j=0}^{W_m - 1} X_{t+j,m}.
\]
Simplifying the left-hand side:
\[
W_m \left( X_{t,m} - X_{t-1,m} \right) \leq \sum_{t'=t}^{t+W_m - 1} X_{t',m}.
\]
This satisfies Constraint~\eqref{eq:5b}.

Therefore, all constraints of Formulation~\eqref{eq:5} are satisfied by any feasible solution of Formulation~\eqref{eq:7}.

\vspace{0.5em}

Next, we show that Formulation~\eqref{eq:7} is strictly tighter by providing a fractional solution that is feasible for Formulation~\eqref{eq:5} but infeasible for Formulation~\eqref{eq:7}.

\textbf{Counterexample:} Consider the same fractional solution used in the proof of Proposition~\ref{pro:2-7}. Let there be $T = \{1, 2, \dots, 10\}$ and a maintenance task $m$ with duration $W_m = 3$. Define the fractional solution:
\[
X_{t,m} = 
\begin{cases}
\frac{1}{3}, & \text{for } t = 1, 2, 3, 6, 7, 8, \\
\frac{1}{4}, & \text{for } t = 4, 5, 9, 10.
\end{cases}
\]
This solution satisfies the constraints of Formulation~\eqref{eq:5}:
\begin{itemize}
    \item \textbf{Constraint~\eqref{eq:5a}}:
    \[
    \sum_{t \in T} X_{t,m} = 6 \times \tfrac{1}{3} + 4 \times \tfrac{1}{4} = 2 + 1 = 3 = W_m.
    \]
    \item \textbf{Constraint~\eqref{eq:5b}}:
    For all $t \in T$, we can verify that:
    \[
    W_m \left( X_{t,m} - X_{t-1,m} \right) \leq \sum_{t'=t}^{t+W_m - 1} X_{t',m}.
    \]
    For example, at $t=1$:
    \[
    3 \left( \tfrac{1}{3} - 0 \right) = 1 \leq \tfrac{1}{3} + \tfrac{1}{3} + \tfrac{1}{3} = 1.
    \]
    Similar calculations can be done for other $t$.
    \item \textbf{Constraint~\eqref{eq:5c}}: All variables are within their bounds $[0, 1]$.
\end{itemize}

However, there cannot be a solution $S_{t,m}$ that satisfies both \textbf{Constraint~\eqref{eq:6a}}:
\[
\sum_{t \in T} S_{t,m} = 1,
\]
and \textbf{Constraint~\eqref{eq:7b}}:
\[
X_{t,m} = \sum_{t'=t - W_m +1}^{t} S_{t',m} \quad \forall t \in T,
\]
because the given $X_{t,m}$ values do not correspond to any $S_{t,m}$ satisfying these constraints. Specifically, the varying values of $X_{t,m}$ cannot be expressed as cumulative sums of $S_{t,m}$ over fixed durations in a way that sums $S_{t,m}$ to 1.

Therefore, this fractional solution is feasible for Formulation~\eqref{eq:5} but infeasible for Formulation~\eqref{eq:7}, proving that Formulation~\eqref{eq:7} is strictly tighter than Formulation~\eqref{eq:5}.
\end{proof}

\begin{cor}
\label{cor:5-6}
Formulation~\eqref{eq:6} is strictly tighter than Formulation~\eqref{eq:5}.
\end{cor}

\section{Conclusions}
\label{sec:conclusions}

In this paper, we have systematically analyzed different MIP formulations for the Generator Maintenance Scheduling problem, focusing on the tightness of their linear relaxations. Through rigorous theoretical comparisons, we demonstrated that Formulation~\eqref{eq:6} consistently provides the tightest linear relaxation among all the models examined. This superiority translates to better dual bound found in the linear relaxation and should be used for solving GMS problems with BnB or cutting planes method. Formulations ~\eqref{eq:4} and \eqref{eq:7} are equivalent in tightness to Formulation~\eqref{eq:6} but require twice the number of variables.

Furthermore, some Formulations explicitly represent both the start and state of maintenance, such as Formulation~\eqref{eq:4}. When that is the case, sometimes we only have to impose the integrality of one of the variables. For example, that is trivial for Formulation~\eqref{eq:7}. In that formulation, if variable $S$ is binary, then $X$ is also binary, even if its integrality was not imposed. However, this fact does not impact the tightness analysis since we are analyzing how the linear relaxation of the formulations compares to it others.

Future work is underway to extend this analysis by incorporating additional constraints, such as precedence and non-coincidence conditions, or by exploring the integration of predictive maintenance elements. In this context, the density of the MIP may impact, and the Formulation~\eqref{eq:7} may be advantageous even though it requires more variables.

\bibliographystyle{ieeetr}
\bibliography{bibliography.bib}






\end{document}